\documentclass[10pt]{amsart}
\usepackage{amssymb}
\usepackage{amsbsy}
\usepackage{amscd}
\usepackage[mathscr]{eucal}
\usepackage{verbatim}
\usepackage{version}
\usepackage{color}
\usepackage[matrix,arrow]{xy}

\def\cal{\mathcal}
\def\Bbb{\mathbb}
\def\frak{\mathfrak}

\newenvironment{NB}{
\color{red}{\bf NB}. \footnotesize 
}{}
\excludeversion{NB}
\newenvironment{NB2}{
\color{blue}{\bf NB}. \footnotesize
}{}

\newcommand{\Pic}{\operatorname{Pic}}

\newcommand{\Coh}{\operatorname{Coh}}

\newcommand{\Ext}{\operatorname{Ext}}
\newcommand{\Hom}{\operatorname{Hom}}

\newcommand{\im}{\operatorname{im}}

\newcommand{\rk}{\operatorname{rk}}

\newcommand{\chr}{\operatorname{char}}

\newcommand{\NS}{\operatorname{NS}}
\newcommand{\coker}{\operatorname{coker}}

\newcommand{\Hilb}{\operatorname{Hilb}}

\newcommand{\tr}{\operatorname{tr}}

\newcommand{\id}{\operatorname{id}}

\font\b=cmr10 scaled \magstep5
\def\bigzerou{\smash{\lower1.7ex\hbox{\b 0}}}

\numberwithin{equation}{section}

\theoremstyle{plain}
 \newtheorem{thm}{Theorem}[section]
 \newtheorem{lem}[thm]{Lemma}
 \newtheorem{prop}[thm]{Proposition}

\theoremstyle{definition}
 \newtheorem{defn}[thm]{Definition}
 
\theoremstyle{remark}
 \newtheorem{rem}[thm]{Remark}

\begin{document}

\title{A note on stable sheaves on Enriques 
surfaces II}
\author{K\={o}ta Yoshioka}
\address{Department of Mathematics, Faculty of Science,
Kobe University,
Kobe, 657, Japan
}
\email{yoshioka@math.kobe-u.ac.jp}

\thanks{
The author is supported by the Grant-in-aid for 
Scientific Research (No.\ 26287007,\ 24224001), JSPS}

\begin{abstract}
We shall give a necessary and sufficient condition for the 
existence of stable sheaves on non-classical Enriques surfaces.
\end{abstract}

\maketitle

\renewcommand{\thefootnote}{\fnsymbol{footnote}}

\section{Introduction}

Let $X$ be an arbitrary Enriques surface over an algebraically 
closed field $k$.
In \cite{Y:Enriques}, we studied the non-emptyness of the 
moduli space of stable sheaves on classical
Enriques surfaces, i.e.,
$K_X \ne 0$ based on results by Kim \cite{Kim1} 
and Nuer \cite{Nuer}.
In this note, we treat the case of non-classical Enriques surfaces, 
i.e., $K_X=0$, and prove the same result holds. 
For a coherent sheaf $E$ on $X$, we define the Mukai vector $v(E)$
of $E$ as
$v(E):=(\rk E,c_1(E),\chi(E)-\frac{\rk E}{2})
\in {\Bbb Z} \times \NS(X) \times {\Bbb Q}$.
For two Mukai vectors $v=(r,L,a)$ and $v'=(r',L',a')$,
$\langle v,v' \rangle:=(L,L')-ra'-r'a \in {\Bbb Z}$ is the Mukai pairing.  
Let $M_H(v)$ be the moduli scheme of stable sheaves $E$
with $v(E)=v$. If $v$ is primitive, then
$M_H(v)$ is projective for a general $H$.
\begin{NB}
\begin{thm}\label{thm:spherical}
For a Mukai vector  $v=(r,L,a)$ with $\langle v^2 \rangle=-2$,
$M_H(v) \ne \emptyset$ if and only if 
$L \equiv D +\frac{r}{2}K_X \mod 2$ where $D$ is a nodal cycle. 
\end{thm}
\end{NB}

\begin{thm}\label{thm:exist:nodal}
Let $X$ be an Enriques surface over $k$.
We take $r,s \in {\Bbb Z}$ $(r>0)$ and $L \in \NS(X)$ such
that $r+s$ is even.
Assume that 
$\gcd(r,L,\frac{r+s}{2})=1$, i.e., the Mukai vector $(r,L,\frac{s}{2})$
is primitive.
Then
$M_H(r,L,\tfrac{s}{2}) \ne \emptyset$ for a general $H$
if and only if
\begin{enumerate}
\item
$\gcd(r,L,s)=1$ and $(L^2)-rs \geq -1$ or 
\item 
$\gcd(r,L,s)=2$ and $(L^2)-rs \geq 2$ 
or 
\item
$\gcd(r,L,s)=2$,
$(L^2)-rs =0$ and $L \equiv \frac{r}{2}K_X \mod 2$ or
\item
$(L^2)-rs =-2$,
$L \equiv D+\frac{r}{2}K_X \mod 2$, where 
$D$ is a nodal cycle, that is, $(D^2)=-2$ and $H^1({\cal O}_D)=0$.
\end{enumerate}
\end{thm}

\begin{rem}
If $(L,H')>0$ for an ample divisor $H'$, then
the same claim holds for $r=0$.
\end{rem}

\begin{NB}
If $(L,H')>0$, $(L^2)=-2$ and $L \equiv D \mod 2$ for a nodal cycle, then
there is a stable locally free sheaf $E$ of rank 2 with $c_1(E)=L+K_X$.
Since $\chi(E)=1$ and $(c_1(E),H)>0$,
$H^0(E) \ne 0$. Then there is an effective divisor $C$ such that
$E(-C)$ has a section and $H^0(E(-C-C'))=0$ for all effective 
divisor $C'$. Then 
as we shall see,  $E(-C)$ fits in an exact sequence
$$
0 \to {\cal O}_X(K_X) \to E(-C) \to {\cal O}_X(D') \to 0
$$
where $D'$ is a nodal cycle.
Hence $c_1(E)-2C=D'+K_X$, which implies
$L=2C+D'$. Therefore $L$ is effective.
\end{NB}

We note that $(L^2)-rs=\langle v(E)^2 \rangle \geq -2$ for any stable sheaves
$E$ with $v(E)=(r,L,\frac{s}{2})$.
Since $X$ is liftable to a field of characteristic 0 by \cite[Thm. 4.10]{Lie}
or \cite[Thm. 5.7]{ESB},
the non-emptyness of the moduli space is a consequence
of the result for the case of characteristic 0, if
$(L^2)-rs \geq 0$.   
Indeed $\NS(X) \mod K_X$ is locally constant 
(\cite[Prop. 4.4]{Lie}, \cite[Cor. 4.3]{ESB}),
which shows that 
we have a family of moduli spaces over a polarized family of Enriques surfaces.
So we treat the case where $(L^2)-rs=-2$. 
\begin{NB}
Over the generic fiber of a family of Enriques surfaces, we have
basis $\{ L_i \}$ of the Picard group.
By extending them to central fiber,
we have a basis of $\NS(X) \mod K_X$.
For each $L_i$, we have a family of moduli spaces, which is 
dominant. By the projectivity of the moduli spaces,
the family is surjective over the base.
\end{NB}

For the Mukai vector $v_0=v({\cal O}_X \oplus {\cal O}_X(K_X)-k_x)=(2,K_X,0)$,
$M_H(v_0)=X$ and we have a Fourier-Mukai transform
$\Phi_{X \to X}^{\cal E}:{\bf D}(X) \to {\bf D}(X)$
which acts on the set of Mukai vectors as
$(r,L+\frac{r}{2}K_X,\frac{s}{2}) \mapsto
(s,-(L+\frac{s}{2}K_X)+s K_X,\frac{r}{2})$
(Proposition \ref{prop:FM}).
As in \cite{Y:Enriques}, 
by using \cite[Cor. 4.5]{Y:Enriques}
we can reduce Theprem \ref{thm:exist:nodal} (iv)
to the following result which was proved by Kim \cite{Kim1}
if $X$ is classical.

\begin{thm}\label{thm:exceptional}
For a Mukai vector $v=(2,L,\frac{s}{2})$ with $(L^2)-2s=-2$,
there is a stable sheaf $E$ such that $v(E)=v$ if and only if 
$L \equiv D +K_X \mod 2$ where $D$ is a nodal cycle. 
\end{thm}
By modifying Kim's proof \cite{Kim1},
we give a proof which works for non-classical case.

\subsection{Preliminaries}
Throughout this note, $X$ denotes an Enriques surface over $k$.
By the work of Bombieri and Mumford \cite{BM},
there are two classes of Enriques surfaces. 
\begin{enumerate}
\item[1.]
{\bf Classical Enriques surface}.
$K_X \ne 0$, $H^1({\cal O}_X)=H^2({\cal O}_X)=0$ and
$\Pic(X)$ is smooth.
\item[2.]
{\bf Non-classical Enriques surface}.
$K_X=0$ and 
$\dim H^1({\cal O}_X)=\dim H^2({\cal O}_X)=1$. 
In this case, $\chr(k)=2$ and $\Pic(X)$ is non-reduced. 
\end{enumerate}
The torsion free quotient of $\NS(X)$ is unimodular. Thus
$\NS(X)/{\Bbb Z}K_X \cong U \oplus E_8$, where $U$ is
the hyperbolic lattice \cite{I}. 
For more detail of Enriques surfaces, we refer
\cite{BM}, \cite{CD}, \cite{ESB}, \cite{Lie}.
In this note, $D=D'$ means
$D=D'$ in $\NS(X)=\Pic(X)$, that is,
${\cal O}_X(D) \cong {\cal O}_X(D')$.

For the proof of Theorem \ref{thm:exceptional}, 
we use the following vanishing theorem in \cite{CD}.
\begin{lem}\label{lem:vanishing}
Let $D$ be an effective divisor which is nef. 
\begin{enumerate}
\item[(1)]
If $(D^2)>0$, then
Then $H^1({\cal O}_X(-D))=H^1({\cal O}_X(-D+K_X))=0$.
\item[(2)]
If $(D^2)=0$ and $D$ is primitive in $\NS(X)$, then
$H^1({\cal O}_X(-D+K_X))=0$.
\end{enumerate}
\end{lem}

\begin{proof}
(1)
Since $(D,H)=(D-K_X,H)>0$ and $\chi({\cal O}_X(D-K_X))=
\chi({\cal O}_X(D))>0$,
$D-K_X$ is also effective.
Hence we shall prove $H^1({\cal O}_X(-D))=0$.
If $(D^2)>0$, then $H^1({\cal O}_X(-D))=0$ by
\cite[Cor. 3.1.3]{CD} (see also \cite[Thm. 3]{Mukai} 
and \cite[Cor. 8]{SB}).
\begin{NB}
If there is a non-trivial extension
$0 \to {\cal O}_X \to E \to {\cal O}_X(D) \to 0$,
then $E$ is not stable by $\langle v(E)^2 \rangle=-(D^2)-4<-2$.
Hence there is an exact sequence
$0 \to {\cal O}_X(A) \to E \to I_Z(B) \to 0$ such that
$(B-A,H)<0$ for an ample divisor $H$, 
where $(A,B)+\deg Z=c_2(E)=0$ and $A+B=c_1(E)=D$.
Then $((B-A)^2)=(D^2)-4(A,B)=(D^2)+4 \deg Z>0$.
Hence $A-B$ is effective.
In particular $(D-2B,L)=(A-B,D) \geq 0$.
Since $(A,B)=-\deg Z \leq 0$, $(D,B) \leq (B^2)$.
By Hodge index theorem,
$(B,L)^2 \geq (B^2)(D^2) \geq (D,B)(D^2)$.
If $\phi:{\cal O}_X \to E \to I_Z(B)$ is a zero map,
then ${\cal O}_X$ is a subsheaf of ${\cal O}_X(A)$, which implies
$A=0$. Then $A-B=2A-L=-L$, which is a contradiction. 
Therefore $\phi \ne 0$.
If $\phi$ is an isomorphism, then
the extension splits. Hence $B$ is effective.
In particular $(B,L) \geq 0$.
If $(B,L)=0$, then $(B^2)<0$. Hence $(L,B) \leq (B^2)<0$, which
is a contradiction. Thus $(B,L)>0$.
Then $(B,L) \geq (L^2) \geq 2(B,L) \geq 0$, which does not hold. 
Therefore $H^1({\cal O}_X(-D))=0$.
\end{NB}

(2)
If $(D^2)=0$, then
by \cite[Cor. 5.7.2]{CD}, 
$\dim H^0({\cal O}_X(D))=1$.
Hence $H^1({\cal O}_X(-D+K_X))=H^1({\cal O}_X(D))^{\vee}=0$.
\end{proof}

\begin{lem}[{\cite[Cor. 2.7.1]{CD}}]\label{lem:phi}
For an effective divisor $D$ with $(D^2)>0$,
there is an effective and isotropic divisor $f$ with
 $0<(D,f) \leq \sqrt{(D^2)}$.
\end{lem}

\begin{proof}
By the action of the Weyl group associated to
the root system of smooth rational curves on $X$,
we may assume that $D$ is nef. 
By \cite[Cor. 2.7.1]{CD},
there is a divisor $f$ such that $(f^2)=0$ and
$0 \leq (D,f) \leq \sqrt{(D^2)}$. 
By the Hodge index theorem,
$(D,f)>0$.
$(f^2)=0$ implies $f$ or $K_X-f$ is effective.
If $K_X-f$ is effective, then $(D,-f) \geq 0$ by the nefness of $D$.
Hence $f$ is effective. 
\begin{NB}
Even if $D$ is a non-nef effective divisor,
there is an effective divisor $f$ with $0<(D,f) \leq \sqrt{(D^2)}$.
Indeed, there is a Weyl group action $w$ with
$w(D)$ is a nef effective divisor.
Then there is an effective and isotropic divisor $f'$ such that
$0<(w(D),f') \leq \sqrt{(w(D)^2)}=\sqrt{(D^2)}$.
Since $f=w(f')$ is effective and $(w(D),f')=(D,w(f'))$,
we get the claim.

Note:
For $0 \geq x \geq (C,D)$, $((D+xC)^2)=(D^2)+2x((C,D)-x) \geq (D^2)>0$.
Hence $(D,H)>0$ if and only if $(D+(C,D)C,H)>0$. 
\end{NB}
\end{proof} 

\begin{NB}
\begin{rem}
We do not need $(D,f)>0$ for the proof of Proposition \ref{prop:MO},
although we used it.
\end{rem}
\end{NB}

\begin{lem}[{\cite[Thm. 1.7]{A}}]\label{lem:nodal}
Let $D$ be an effective divisor on $X$ such that $(D^2)=-2$.
\begin{NB}
Let $D$ be an effective divisor on $X$ such that $(D^2)=-2$.
\end{NB}
\begin{enumerate}
\item[(1)]
If $D$ is nodal, i.e., $H^1({\cal O}_D)=0$, then
\begin{enumerate}
\item
${\cal O}_D$ is stable for any $H$.
\item
${\cal O}_D(K_X) \cong {\cal O}_D$. 
\item
$\dim H^0({\cal O}_D)=1$, $\dim H^0({\cal O}_D(D))=0$
and $\dim H^1({\cal O}_D(D))=1$.
\end{enumerate}
\item[(2)]
$D$ is nodal if and only if
$(C^2)<0$ for all decomposition  
$D=C+C'$ by effective divisors.
\end{enumerate}
\end{lem}

\begin{proof}
(1)
By $\chi({\cal O}_D)=-(D^2)/2=1$,
$\dim H^0({\cal O}_D)=1$.  
For any quotient $\phi:{\cal O}_D \to E$ with $\dim E=1$,
there is a divisor $C$ with $E={\cal O}_C$.
Hence $\chi({\cal O}_C)=\dim H^0({\cal O}_C) \geq 1$.
Then $\chi({\cal O}_D)/(D,H)<\chi({\cal O}_C)/(C,H)$
if $\phi$ is not isomorphic.
Therefore ${\cal O}_D$ is stable.
Since $K_X$ is numerically trivial,
${\cal O}_D(K_X)$ is also stable. 
Since $\chi({\cal O}_D,{\cal O}_D)=-(D^2)=2$,
$\Hom({\cal O}_D,{\cal O}_D(K_X)) \ne 0$,
which implies ${\cal O}_D \cong {\cal O}_D(K_X)$.
Since ${\cal O}_D(K_D)={\cal O}_D(D)$,
we also have the remaining claim.

(2)
Assume that $D$ is nodal. 
By the exact sequence
\begin{equation}\label{eq:O_D}
0 \to {\cal O}_X \to {\cal O}_X(D) \to {\cal O}_D(D) \to 0,
\end{equation}
we have $\dim H^0({\cal O}_X(D))=1$.
Hence $D=C+C'$ in $\NS(X)$ implies
$D=C+C'$ as a divisor.
Thus ${\cal O}_C$ is a quotient of ${\cal O}_D$,
which implies $H^1({\cal O}_C)=0$ and $(C^2)=-2\chi({\cal O}_C) \leq -2$.

Conversely if $(C^2)<0$ for all decomposition 
$D=C+C'$ by effective divisors, then \cite[Thm.1.7]{A}
implies $D$ is nodal.
\begin{NB}
Since $(D^2)<0$, there is a smooth rational curve $C$
with $(D,C)<0$. Then $D-C=0$ or 
$D-C$ is effective and
$0 > ((D-C)^2)=(D^2)-2(D,C)+(C^2)=(D^2)+2(-(D,C)-1) \geq (D^2)$.
Hence $(D,C)=-1$ and $((D-C)^2)=-2$.
Since $D-C$ also satisfies the assumption,
inductively we have a sequence of effective divisors
$D_0=D,D_1,D_2,...,D_s$ such that
$C_i:=D_i-D_{i+1}$ are smooth rational curves, where $D_{s+1}=0$.
By the exact sequence
$$
0 \to {\cal O}_{C_i}(-1) \to {\cal O}_{D_i} \to {\cal O}_{D_{i+1}} \to 0,
$$
we have $H^1( {\cal O}_{D_i} ) \cong H^1( {\cal O}_{D_{i+1}})$.
Hence $H^1({\cal O}_D)=0$. 

Conversely if there is an effective decomposition $D=C+C'$ 
with
$(C^2) \geq 0$, then
$\chi({\cal O}_C)=-(C^2)/2 \leq 0$.
Hence $H^1({\cal O}_C) \ne 0$.
Since $H^1({\cal O}_D) \to H^1({\cal O}_C)$ is surjective,
$H^1({\cal O}_D) \ne 0$. 
\end{NB}
\begin{NB}
We note that the arithmetic genus $p(C)$ is
$(C^2)/2+1$ for an effective divisor $C$ on $X$.
Thus $p(C) \leq 0$ if and only if $(C^2)<0$. 
\end{NB}
\end{proof}

\begin{NB}
We can also prove that ${\cal O}_D={\cal O}_D(K_X)$.
Indeed we have $\Ext^1({\cal O}_{D_i},{\cal O}_{C_i}(-1))=0$
by
$0=H^0({\cal O}_{C_i}(-1+(D_i,C_i))) \to 
\Ext^1({\cal O}_{D_i},{\cal O}_{C_i}(-1))
\to \Ext^1({\cal O}_X,{\cal O}_{C_i}(-1))=0$.....

\begin{rem}
${\cal O}_{D_i}$ are stable with $\chi({\cal O}_{D_i})=1$
by $H^1({\cal O}_{D_i})=0$.
Since $\chi({\cal O}_{D_i},{\cal O}_{D_i})=-(D_i^2)=2$,
$\Hom({\cal O}_{D_i},{\cal O}_{D_i}(K_X)) \ne 0$, which implies
${\cal O}_{D_i} \cong {\cal O}_{D_i}(K_X)$.
\end{rem}
\end{NB}

The following result is well-known.
\begin{lem}\label{lem:spherical}
Let $E$ be a stable sheaf of $\langle v(E)^2 \rangle=-2$.
Then $E \cong E(K_X)$.
If $E$ is torsion free, then it is locally free. 
\end{lem}

\begin{NB}
\begin{proof}
Since $\chi(E,E)=2$ and $\Hom(E,E)=k$
implies there is a non-zero homomorphism
$\phi:E \to E(K_X)$. Since $K_X$ is numerically trivial,
stability of $E$ implies $\phi$ is isomorphic.
Then $\Ext^1(E,E)=0$.
If $E$ is torsion free, then 
$\Ext^1(E,E)=0$ implies $E$ is locally free.
\end{proof}
\end{NB}

\begin{prop}\label{prop:spherical2}
Let $E$ be a torsion free sheaf of rank 2 such that
$E$ is simple and $\langle v(E)^2 \rangle=-2$.
Then it is $\mu$-stable with respect to any polarization.  
\end{prop}

\begin{proof}
Let $H$ be an ample divisor.
Assume that there is an exact sequence
\begin{equation}\label{eq:unstable}
0 \to E_1 \to E \to E_2 \to 0
\end{equation}
such that $E_1$ and $E_2$ are torsion free sheaves of rank 1 and
$(c_1(E_1)-c_1(E_2),H) \geq 0$.
We note that $c_1(E_2)-c_1(E_1)$ is not numerical trivial,
by $4 \nmid \langle v(E)^2 \rangle$.
Since $E$ is simple,
$\Hom(E_2,E_1)=0$.
If $\Hom(E_1,E_2(K_X))^{\vee}=\Ext^2(E_2,E_1) \ne 0$, then
$c_1(E_2)-c_1(E_1)+K_X$ is effective, since $c_1(E_2)-c_1(E_1)+K_X \ne 0$. 
On the other hand, $(c_1(E_1)-c_1(E_2),H) \geq 0$ implies
$c_1(E_2)-c_1(E_1)+K_X$ is not effective. Hence
we get $\Ext^2(E_2,E_1)= 0$.
By the simpleness of $E$,
\eqref{eq:unstable} does not split, and hence
$\Ext^1(E_2,E_1) \ne 0$.
Therefore $\langle v(E_1),v(E_2) \rangle >0$.
Since
$\langle v(E_1)^2 \rangle,\langle v(E_2)^2 \rangle \geq -1$,
we see that
$$
-2=\langle v(E)^2 \rangle=
\langle v(E_1)^2 \rangle+\langle v(E_2)^2 \rangle
+2\langle v(E_1),v(E_2) \rangle >-2,
$$
which is a contradiction.
Therefore $E$ is $\mu$-stable with respect to $H$.
\end{proof}


\section{Proof of Theorem \ref{thm:exceptional}}

\begin{prop}\label{prop:h^1}
For a nodal cycle $D$,  we take a non-trivial extension
\begin{equation}\label{eq:D}
0 \to {\cal O}_X(K_X) \overset{\phi}{\to} E \to {\cal O}_X(D) \to 0.
\end{equation}
\begin{enumerate}
\item[(1)]
$\dim H^0(E)=\dim H^0(E(K_X))=1$.
\item[(2)]
$E$ is $\mu$-stable.
\end{enumerate}
\end{prop}

\begin{NB}
Since $H^2(E)=H^2(E(K_X))=0$, we also have
$H^1(E)=H^1(E(K_X))=0$.
\end{NB}

\begin{proof}
(1)
By the exact sequence \eqref{eq:O_D}, we have 
$\dim H^0({\cal O}_X(D))=1$ and $H^0({\cal O}_X(D+K_X))=H^0({\cal O}_X(K_X))$.
We also have $\dim H^1({\cal O}_X(D))=1$ by
$H^2({\cal O}_X(D))=0$ and $\chi({\cal O}_X(D))=0$.
If $K_X \ne 0$, then $H^0(E) \cong H^0({\cal O}_X(D))$ and
$H^0({\cal O}_X) \cong H^0(E(K_X))$ implies the claim.

Assume that $K_X=0$. Then
$\Ext^i({\cal O}_D(D),{\cal O}_X)=H^{i-1}({\cal O}_D)$ 
and we have an exact sequence 
\begin{equation}
\begin{CD}
0 @>>> \Hom({\cal O}_X(D),{\cal O}_X) @>>> 
H^0({\cal O}_X) @>>> H^0({\cal O}_D)\\
@>>> \Ext^1({\cal O}_X(D),{\cal O}_X) @>>> \Ext^1({\cal O}_X,{\cal O}_X)
@>>> H^1({\cal O}_D).
\end{CD}
\end{equation}
Hence $\varphi:\Ext^1({\cal O}_X(D),{\cal O}_X(K_X)) \to
 \Ext^1({\cal O}_X,{\cal O}_X(K_X))$ is an isomorphism.
By the commutative diagram
\begin{equation}
\begin{CD}
k \cong \Hom({\cal O}_X(D),{\cal O}_X(D)) @>{\delta_1}>> 
\Ext^1({\cal O}_X(D),{\cal O}_X(K_X))\\
@VVV @VV{\varphi}V\\
k \cong \Hom({\cal O}_X,{\cal O}_X(D)) @>{\delta_2}>> 
\Ext^1({\cal O}_X,{\cal O}_X(K_X)),
\end{CD}
\end{equation}
$\delta_2$ is an isomorphism,
where $\delta_1$ is the connecting homomorhism
which is non-trivial by the construction of $E$. 
Hence we have 
$\dim H^0(E)=1$ by $\dim H^0({\cal O}_X(K_X))=1$.

(2)
The following argument is a modification of
(iii) $\to$ (ii) of \cite[Thm. 3.4]{Kim1}.
Assume that there is an exact sequence 
\begin{equation}\label{eq:A-B}
0 \to {\cal O}_X(A) \to E \to I_Z(B) \to 0
\end{equation}
such that $(A,H) \geq (B,H)$.
Since $D=A+B-K_X$ is effective,
$(A+B,H)>0$, which implies $(A,H)>0$.
Since $\coker({\cal O}_X(K_X) \to E)$ is torsion free and
${\cal O}_X(A) \not \cong {\cal O}_X(K_X)$,
${\cal O}_X(K_X) \to E \to I_Z(B)$ is injective.
Thus $B-K_X$ is effective or $B=K_X$.
If $B=K_X$, then $Z=\emptyset$ and
the exact sequences \eqref{eq:D} and \eqref{eq:A-B} split.
So $B-K_X$ is effectieve.   
Since 
$$
1=\chi(E)=\frac{(A^2)}{2}+1+\frac{(B^2)}{2}+1-\deg Z
$$
 and 
$-2=((A+B)^2)$, $\deg Z=-(A,B) \geq 0$.
If $(A,B)<0$, then 
$$
((A-B)^2)=((A+B)^2)-4(A,B)=-2-4(A,B) \geq 2,
$$
which implies that $\chi({\cal O}_X(A-B-K_X))=((A-B)^2)/2+1>0$.
We set $C:=A-B-K_X$ and $C':=2B$.
By $(C,H) \geq 0$ and the ampleness of $H$,
$C$ is effective with $(C^2) \geq 0$.
Since $C'$ is effective,
we have a decomposition $D=C+C'$
which contradicts
with the assumption of $D$ by Lemma \ref{lem:nodal}.
Therefore $(A,B)=0$.

Then $\chi({\cal O}_X(A-B))=((A-B)^2)/2+1=0$.
\begin{NB}
If $B-A$ is effective, then $B-A \ne 0$ by $((B-A)^2) =-2$, and
$(B-A,H)>0$.
\end{NB}
We first prove that $A-K_X$ is effective.
If $H^0({\cal O}_X(A-B))=0$, then
$H^1({\cal O}_X(A-B))=0$, which implies that
$E \cong {\cal O}_X(A) \oplus {\cal O}_X(B)$
and \eqref{eq:D} is the exact sequence 
$$
0 \to {\cal O}_X(K_X) \overset{\phi}{\to} {\cal O}_X(A) \oplus {\cal O}_X(B) 
\to {\cal O}_X(D) \to 0.
$$
In particular $H^0({\cal O}_X(A-K_X)) \ne 0$ by
the effectivity of $B-K_X$ and the torsion freeness of $\coker \phi$.
Since $(A-K_X,H) \geq (B-K_X,H)>0$,
$A-K_X$ is also effective.
\begin{NB}
Another argument:
Hence $D-K_X=(A-K_X)+(B-K_X)$ is a decomposition by effective divisors.
Then as in the argument below, we get a contradiction. 
 On the other hand, $\dim H^0(E(K_X))=1$ implies
$H^0({\cal O}_X(A-K_X))=0$, which is a contradiction.
\end{NB}
If $H^0({\cal O}_X(A-B)) \ne 0$, then 
$A-K_X=(A-B)+(B-K_X)$ is effective.

Hence we have a decomposition
$D-K_X=A+B=(A-K_X)+(B-K_X)$ by effective divisors. 
Since $(A^2)+(B^2)=-2$, $(A^2) \geq 0$ or $(B^2) \geq 0$.
If $(A^2) \geq 0$, then $A$ is also effective.
Thus we have a decomposition 
$D=A+(B-K_X)$ by effective divisors $A$ and $B-K_X$, which
contradicts with 
Lemma \ref{lem:nodal}.
If $(B^2) \geq 0$, similarly we get a contradiction.
Therefore $E$ is $\mu$-stable.
\end{proof}

\begin{rem}\label{rem:h^1}
By using (1), we can directly show that $E$ is simple,
which gives another proof of (2) by Propisition \ref{prop:spherical2}.
\begin{NB}
Since $H^0({\cal O}_X(-D+K_X))=0$ and 
\eqref{eq:D} does not split,
$\Hom({\cal O}_X(D),E)=0$.
Then $\psi_{|{\cal O}_X(K_X)} \ne 0$ 
for a non-zero homomorphism $\psi:E \to E$.
By (1), $\psi$ induces a non-zero homomorphism of
${\cal O}_X(K_X)$. We set
$\Psi_{|{\cal O}_X(K_X)}=\lambda \id_{{\cal O}_X(K_X)}$,
$\lambda \in k^*$.
Then $\psi':=\psi-\lambda \id_{E}$
must be zero, since $\psi'_{|{\cal O}_X(K_X)}=0$.
\end{NB} 
\end{rem}

\begin{prop}[{cf. \cite[Thm. 3.3]{Kim1}}]\label{prop:MO}
Let $E$ be a $\mu$-semi-stable simple sheaf of rank 2 such that
$\langle v(E)^2 \rangle=-2$.
Assume that 
$H^0(E) \ne 0$ and $H^0(E(-C))=0$
for all effective divisor $C$.
Then 
$E$ fits in an exact sequence
\begin{equation}\label{eq:M0}
0 \to {\cal O}_X(K_X) \to E \to {\cal O}_X(D) \to 0
\end{equation}
such that
$D$ is a nodal cycle.
\end{prop}

\begin{proof}
We note that $E$ is a $\mu$-stable locally free sheaf
with respect to any ample divisor $H$
such that $E(K_X) \cong E$ by Lemma \ref{lem:spherical} and
Proposition \ref{prop:spherical2}.
We have an exact sequence
\begin{equation}\label{eq:M}
0 \to {\cal O}_X(K_X) \to E \to I_Z(D) \to 0
\end{equation}
where $D=c_1(E)-K_X$ and
$4\deg Z-(D^2)=2$.
We note that $(D^2) \geq -2$.
By the stability of $E$, we also have $(D,H)>0$, which shows
that $H^2(I_Z(D))=0$.

We first prove that $(D^2)=-2$ and $\deg Z=0$.
Assume that $(D^2) \geq 0$, that is,
$\deg Z>0$ and $(D^2) \geq 2$.
Then $(D,H)>0$ implies that $D$ is effective.
By Lemma \ref{lem:phi},
there is an effective divisor $f$ such that
$(f^2)=0$ and $0<(D,f) \leq \sqrt{(D^2)}$.
\begin{NB}We can choose $(D,f)>0$.\end{NB}
Then for  $0 \leq x \leq 1$,
\begin{equation}
((D-xf)^2)=(D^2)-2x(D,f) \geq (D^2)-2(D,f)\geq 
\sqrt{(D^2)}(\sqrt{(D^2)}-2).
\end{equation}
If $(D^2) \geq 4$, then obviously $(D^2)-2(D,f) \geq 0$.
If $(D^2)=2$, then $(D,f) \leq 1$, which also implies
$(D^2)-2(D,f) \geq 0$.
Since $D \not \in  {\Bbb Q}f$, we see that $(D-xf,H)>0$.
Hence $D-f$ is an effective divisor.
We note that 
$$
\chi(E(-f))=\chi(E)-(D,f)=\deg Z+1-(D,f)
\geq \deg Z+1-\sqrt{4 \deg Z-2}.
$$
Since $(\deg Z+1)^2-(4 \deg Z-2)=(\deg Z-1)^2+2>0$,
we have $\chi(E(-f))>0$, which implies $H^0(E(-f)) \ne 0$ or
$\Hom(E,{\cal O}_X(K_X+f)) \ne 0$.
By our choice of $E$, we have $\Hom(E,{\cal O}_X(K_X+f)) \ne 0$.
Then there is an exact sequence
 $$
0 \to {\cal O}_X(D-f+C) \to E \to I_{Z'}(f+K_X-C) \to 0,
$$
where $C=0$ or $C$ is an effective divisor.
Since $D-f$ is effective, 
this case also does not occur.
Therefore we get $(D^2)<0$. 
Then $\deg Z=0$ and $(D^2)=-2$.
\begin{NB}
Old argument:
We claim that $D$ is nef.
Indeed if $(D,C)<0$ for an irreducible curve $C$,
then $(C^2)=-2$.
\begin{NB2}
By $1 \geq \chi({\cal O}_C)=-(C^2)/2$,
$(C^2)=-2$.
\end{NB2}
Since $C$ is a fixed component of $|D|$ and
$D \ne C$, we have $|D-C| \not= \emptyset$.
\begin{NB2}
If $D=C$, then $(D^2)=-2$.
\end{NB2}
\begin{NB2}
Another argument:
We claim that $(D-xC,H)>0$ for $0 \leq x \leq 1$.
Indeed if $(D-xC,H)=0$ for some $x \in [0,1]$,
then the Hodge index theorem implies that
$0 \geq ((D-xC)^2)=(D^2)-2x((D,C)+x) \geq (D^2) \geq 2$ by
$(D,C) \leq -1$.  

Since $((D-C)^2) \geq (D^2) \geq 2$ and $(D-C,H)>0$,
$D-C$ is effective. Thus 
$|D-C| \not= \emptyset$.
\end{NB2}
We have 
$$
\chi(I_Z(D-C))=\chi(I_Z(D))-(D,C)-1 \geq \chi(I_Z(D))
=\deg Z>0.
$$
\begin{NB2}
$((D-C)^2)/2=(D^2)-(D,C)-1$ and
$\chi(I_Z(D))=(D^2)/2+1-\deg Z=(2\deg Z-1)+1-\deg Z=\deg Z$.
\end{NB2}
Hence 
$$
\chi(E(-C))=\chi({\cal O}_X(K_X-C))+\chi(I_Z(D-C))=
\chi(I_Z(D-C))>0.
$$
Hence $H^0(E(-C)) \ne 0$ or
$H^2(E(-C)) \ne 0$.
By our choice of $E$,
we have $\Hom(E,{\cal O}_X(C+K_X))=
H^2(E(-C))^{\vee} \ne 0$.
Then we have an exact sequence
$$
0 \to {\cal O}_X(D-C+C') \to E \to I_{Z'}(C+K_X-C') \to 0,
$$
where $C'=0$ or $C'$ is an effective divisor.
Since $D-C$ is effective, 
this case also does not occur.
Therefore $D$ is nef.

Let $f$ be an effective divisor of $(f^2)=0$.
Then $f=f_0+R$, where $R$ is effective and
$f_0$ is an isotropic nef and effective divisor.
\begin{NB2}
$R$ consists of $(-2)$-curves.
\end{NB2}
Since $(D,f)=(D,f_0)+(D,R) \geq (D,f_0)$,
we find an isotropic nef and effective divisor $f$
such that $0<(D,f) \leq \sqrt{(D^2)}$
by \cite[Cor. 2.7.1]{CD}.
\begin{NB2}
If $(D,f)=0$, then 
$(D^2) \leq 0$.
Assume that $(D,f)>0$. $(f^2)=0$ implies $f$ or $-f$ is effective.
If $-f$ is effective, then $(D,-f) \geq 0$ by the nefness of $D$.
Hence $f$ is effective. 
\end{NB2}
\begin{NB2}
Even if $D$ is a non-nef effective divisor,
there is an effective divisor $f$ with $0<(D,f) \leq \sqrt{(D^2)}$.
Indeed, there is a Weyl group action $w$ with
$w(D)$ is a nef effective divisor.
Then there is an effective and isotropic divisor $f'$ such that
$0<(w(D),f') \leq \sqrt{(w(D)^2)}=\sqrt{(D^2)}$.
Since $f=w(f')$ is effective and $(w(D),f')=(D,w(f'))$,
we get the claim. 
\end{NB2}
We note that $H+nf$ is ample for $n \geq 0$.
Since $(D-2f,H+nf)=(D-2f,H)+n(D,f)>0$ for $n \gg 0$,
$H^2(E(-f))=0$ by the stability of $E$ with respect to
$H+nf$ (Proposition \ref{prop:spherical2}).
We note that 
$$
\chi(E(-f))=\chi(E)-(D,f)=\deg Z+1-(D,f)
\geq \deg Z+1-\sqrt{4 \deg Z-2}.
$$
Since $(\deg Z+1)^2-(4 \deg Z-2)=(\deg Z-1)^2+2>0$,
we have $\chi(E(-f))>0$, which implies $H^0(E(-f))>0$.
Therefore we get $(D^2)<0$. 
Then $\deg Z=0$ and $(D^2)=-2$.
\end{NB}

Since $E$ is stable, we have
$H^1({\cal O}_X(D))=H^1({\cal O}_X(K_X-D))^{\vee} \ne 0$.
Since $\chi({\cal O}_X(D))=0$ and $(D,H)>0$,
$D$ is an effective divisor.
Assume that $D$ has a decomposition
$D=C+C'$ such that $C$ and $C'$ are effective and
$(C^2) \geq 0$.
Then we have a decomposition $C=C_1+R$
such that $C_1$ is an effective, nef divisor with
$(C_1^2)=(C^2)$ and $R$ is an effective divisor consisting of
$(-2)$-curves.
If $(C_1^2)>0$, then Lemma \ref{lem:nodal} implies
$H^1({\cal O}_X(-C_1+K_X))=0$,
which implies $H^0(E(-C_1)) \ne 0$.
If $(C_1^2)=0$, then there is a decomposition
$C_1=C_2+C_3$ by effective divisors such that
$C_2$ is primitive in $\NS(X)$ and $(C_2^2)=0$.
In this case, we also have $H^1({\cal O}_X(-C_2+K_X))=0$ by
Lemma \ref{lem:nodal}, 
and hence $H^0(E(-C_2)) \ne 0$.
Therefore 
$D$ is a nodal cycle.
\end{proof}

\begin{rem}
There are many locally free unstable sheaves $E$ with
$\langle v(E)^2 \rangle=-2$.
Thus some conditions may be missing  in the classification of \cite[Thm. 3.3]{Kim1}.
\end{rem}

\begin{NB}
Let $D$ be a primitive and isotropic divisor which is 
nef and effective. Then $H^1({\cal O}_X(-D))=H^1({\cal O}_X(-D+K_X))=0$.

Since $(D,H)=(D+K_X,H)>0$ and $\chi({\cal O}_X(D+K_X))=\chi({\cal O}_X)=1$,
$D+K_X$ is also effective.
By \cite[Cor. 5.7.2]{CD}, $\dim H^0({\cal O}(D))=
\dim H^0({\cal O}_X(D+K_X))=1$.
Hence $H^1({\cal O}_X(D))=H^1({\cal O}_X(D+K_X))=0$.

Let $D$ be a nef divisor which is 
effective and $(D^2)>0$. Then $H^1({\cal O}_X(-D))=0$ by
\cite[Cor. 3.1.3]{CD}.

\begin{lem}[{\cite[Prop. 4.3.3]{CD}}]
Let $D$ be an irreducible effective dvisor with $(D^2)>0$.
Then $|D|$ is 1-connected. In particular, $h^0({\cal O}_C)=1$ for
$C \in |D|$.  
\end{lem}
\end{NB}

\begin{NB}
\begin{lem}
Let $E$ be a torsion free sheaf of rank 2 such that
$\Hom(E,E)=\Ext^2(E,E)=k$, $\Ext^1(E,E)=0$.
Then $E$ is $\mu$-semi-stable.
\end{lem}

\begin{proof}
Assume that $E$ is not $\mu$-semi-stable.
Then we have an exact sequence
$$
0 \to E_1 \to E \to E_2 \to 0
$$ 
such that $E_1$ and $E_2$ are torsion freesheaves of rank 1 with
$(c_1(E_1)-c_1(E_2),H)<0$.
Since $E$ is rigid,
$E$ is locally free, and hence, $E_1$ is also locally free.
Then
replacing $E$ by $E \otimes E_1^{\vee}$,
we may assume that $E_1={\cal O}_X$ and
$E_2=I_Z(L)$.
Then $(L^2)=4\deg Z-2$ and
$\chi(I_Z(L))=\chi(I_Z(L),{\cal O}_X)=\deg Z$.
Since $E$ is simple, $\Hom(I_Z(L),{\cal O}_X)=0$.
By $(L,H)<0$,
$\Ext^2(I_Z(L),{\cal O}_X)=0$.
Hence $\deg Z = 0$ and $\Ext^1(I_Z(L),{\cal O}_X)=0$,
which implies $E={\cal O}_X \oplus I_Z(L)$.
Then $E$ is not simple.
Therefore $E$ is $\mu$-semi-stable.
\end{proof}
\end{NB}

{\it Proof of Theorem \ref{thm:exceptional}:

Let $F$ be a stable sheaf with $\langle v(F)^2 \rangle=-2$.
There is a divisor $C$ such that $F(C)$ satisfies the assumption
of  Proposition \ref{prop:MO}.
Then $D:=c_1(F(C))-K_X$ is nodal, which shows that
$c_1(F) \equiv D+K_X \mod 2$.

Conversely  for a Mukai vector $v=(2,L,a)$ with
$L=D+K_X+2C$, there is a $\mu$-stable locally free sheaf
$E$ with $v(E(C))=v$ by Proposition \ref{prop:h^1}. 
\qed

\section{Appendix}

\subsection{Fourier-Mukai transforms}
Let $p_i:X \times X \to X$ $(i=1,2)$ be the projection.
Let $\Delta \subset X \times X$ be the diagonal.
Then we have  
$\Ext_{p_2}^1({\cal O}_X(K_X) \boxtimes {\cal O}_X(K_X),I_\Delta) 
\cong {\cal O}_X$.
\begin{NB}
We have $\Hom_{p_2}({\cal O}_X(K_X) \boxtimes {\cal O}_X(K_X),I_\Delta) =0$.
Hence
$$
\Ext^1_{{\cal O}_{X \times X}}({\cal O}_X(K_X) 
\boxtimes {\cal O}_X(K_X),I_\Delta)
\cong H^0(\Ext^1_{p_2}({\cal O}_X(K_X) \boxtimes {\cal O}_X(K_X),I_\Delta) )=k.
$$
\end{NB}
Let ${\cal E}$ be the universal extension:
\begin{equation}
0 \to I_\Delta \to {\cal E} \to 
{\cal O}_X(K_X) \boxtimes {\cal O}_X(K_X) \to 0.
\end{equation}
Let $\Phi_{X \to X}^{{\cal E}}:{\bf D}(X) \to {\bf D}(X)$
be an integral functor whose kernel is
${\cal E}$:
\begin{equation}\label{eq:FM}
\Phi_{X \to X}^{{\cal E}}(E):=
{\bf R}p_{2*}(p_1^*(E) \overset{\bf L}{\otimes} {\cal E}),\;
E \in {\bf D}(X).
\end{equation} 

\begin{prop}\label{prop:FM}
\begin{enumerate}
\item[(1)]
${\cal E}_{|\{ x\} \times X}$ is stable for all $x \in X$ and
\begin{equation}\label{eq:univ}
\begin{matrix}
\phi:&X & \to & M_H(v_0)\\
& x & \mapsto & {\cal E}_{|\{x \} \times X}
\end{matrix}
\end{equation} 
is an isomorphism.
\item[(2)]
$\Phi_{X \to X}^{{\cal E}}$ is an equivalence. 
In particular 
$\Phi_{X \to X}^{{\cal E}}$ induces 
an automorphism of $K(X)$ which is given by
\begin{equation}\label{eq:(-1)}
\Phi_{X \to X}^{{\cal E}}(E)=\chi(E)({\cal O}_X +{\cal O}_X(K_X))-E,\;
E \in K(X).
\end{equation}
\end{enumerate}
\end{prop}

\begin{proof}
We note that ${\cal E}_{|\{x \} \times X}$ $(x \in X)$
is a non-trivial extension
\begin{equation}\label{eq:univ-ext}
0 \to I_{\{x \}} \to {\cal E}_{|\{x \} \times X}
\overset{\varphi}{\to} {\cal O}_X(K_X) \to 0
\end{equation}
and 
$({\cal E}_{|\{x \} \times X})^{\vee \vee}/{\cal E}_{|\{x \} \times X} 
\cong k_x$.
If ${\cal E}_{|\{x \} \times X}$ $(x \in X)$ is not stable, then
there is a torsion free subsheaf $E_1$
of ${\cal E}_{|\{x \} \times X}$ such that
$\rk E_1= 1$ and $\chi(E_1) \geq 1$. 
Then we see that $\varphi_{|E_1} \ne 0$, which implies 
\eqref{eq:univ-ext} split.  
%
Therefore ${\cal E}_{|\{x \} \times X}$ $(x \in X)$ is stable,
and we get the morphism $\phi$.
Since $\phi$ is injective,
we have $\dim \im \phi=2$.
We also see that $\phi$ is an immersion (Remark \ref{rem:KS}).

Since $\chi({\cal E}_{|\{x \} \times X},{\cal E}_{|\{x \} \times X})=0$ and
$\dim 
\Ext^1({\cal E}_{|\{x \} \times X},{\cal E}_{|\{x \} \times X}) 
\geq \dim \im \phi =2$,
we see that 
$$
\Hom({\cal E}_{|\{x \} \times X},{\cal E}_{|\{x \} \times X}(K_X)) 
=\Ext^2({\cal E}_{|\{x \} \times X},{\cal E}_{|\{x \} \times X})^{\vee} \ne 0,
$$
which implies 
${\cal E}_{|\{x \} \times X}(K_X) \cong {\cal E}_{|\{x \} \times X}$.
We also have $\dim \Ext^1({\cal E}_{|\{x \} \times X},
{\cal E}_{|\{x \} \times X})=2$.
Then the integral functor
$\Phi_{X \to X}^{{\cal E}}$ is an equivalence (\cite{Br:2}).
By the general theory of Fourier-Mukai transform,
$\phi(X)$ is the unique component of 
$M_H(v_0)$. Therefore $\phi$ is an isomorphism. 
\begin{NB}
(cf. \cite[p.405]{PerverseII})
\end{NB}
Since 
$$
{\cal E}={\cal O}_{X \times X}+{\cal O}_X(K_X) \boxtimes {\cal O}_X(K_X)
-{\cal O}_\Delta
$$ 
in $K(X \times X)$,
we have \eqref{eq:(-1)}.
\end{proof}

\begin{rem}\label{rem:DD}
If $K_X=0$, then 
$\Ext^1({\cal O}_X,I_{\{x \}}) \to \Ext^1({\cal O}_X,{\cal O}_X)$
is an isomorphism.
Hence $({\cal E}_{|\{x \} \times X})^{\vee \vee}$
fits in a non-trivial extension
$$
0 \to {\cal O}_X \to ({\cal E}_{|\{x \} \times X})^{\vee \vee} 
\to {\cal O}_X \to 0.
$$
On the other hand,
$({\cal E}_{|\{x \} \times X})^{\vee \vee}=
{\cal O}_X \oplus {\cal O}_X(K_X)$ if $K_X \ne 0$.
\end{rem}

\begin{rem}\label{rem:KS}
By the description of $({\cal E}_{|\{x \} \times X})^{\vee \vee}$
in Remark \ref{rem:DD},
we also have $\Hom({\cal E}_{|\{x \} \times X}),{\cal O}_X(K_X)) \cong k$.
If the Kodaira-Spencer map
$\Ext^1(k_x,k_x) \to 
\Ext^1({\cal E}_{|\{x \} \times X},{\cal E}_{|\{x \} \times X})$
is not isomorphic,
for a subscheme $Z$ of $X$ with
an exact sequence 
$0 \to k_x \to {\cal O}_Z \to k_x \to 0$,
we have a splitting
${\cal E}_{|Z \times X} \cong {\cal E}_{|\{x \} \times X}^{\oplus 2}$ as an
${\cal O}_X$-module.
Since $\Hom({\cal E}_{|\{x \} \times X}),{\cal O}_X(K_X)) \cong k$,
we see that $I_Z=\ker({\cal E}_{|Z \times X},{\cal O}_{Z \times X}(K_X))$
is isomorphic to $I_{\{ x\}}^{\oplus 2}$ as an ${\cal O}_X$-module.
Therefore the Kodaira-Spencer map is injective. 
\end{rem}

\begin{prop}
If $\gcd(r,L,s)=2$, $(L^2)-rs=0$ and $L \equiv \frac{r}{2}K_X
\mod 2$, then $M_H(r,L,\frac{s}{2})$
is smooth of dimension 2, and the universal family
defines a Fourier-Mukai transform.
\end{prop}

\begin{proof}
Let $X' \to {\Bbb F}_q$ be a reduction of $X$ to a finite field ${\Bbb F}_q$
and $M_H(r,L,\frac{s}{2})'$ be the corresponding 
moduli space of stable sheaves on $X'$.
By \cite[Cor. 4.5]{Y:Enriques},
$\# M_H(r,L,\frac{s}{2})'({\Bbb F}_q)=\# X'({\Bbb F}_q)$.
Hence $\dim  M_H(r,L,\frac{s}{2})=2$.
Since $\dim \Ext^1(E,E) \leq 2$ for all $E \in M_H(r,L,\frac{s}{2})$,
$M_H(r,L,\frac{s}{2})$ is smooth of 
$\dim M_H(r,L,\frac{s}{2})=2$.
Then we have $\Hom(E,E(K_X)) \ne 0$, which shows that
$E(K_X) \cong E$. Therefore the universal family defines a Fourier-Mukai
transform.
\end{proof}

\begin{rem}
For a stable locally free sheaf $G$ with $\langle v(G)^2 \rangle=-1$,
we set ${\cal G}:=\ker(G \boxtimes G^{\vee} \to {\cal O}_{\Delta})$.
Then $\Ext_{p_2}^1(p_1^*(G(K_X)),{\cal G}) \cong G(K_X)^{\vee}$ and
the universal extension 
$$
0 \to {\cal G} \to {\cal E} \to G(K_X) \boxtimes G(K_X)^{\vee} \to 0
$$
defines a universal family of stable sheaves
with Mukai vector $2\rk G v(G)-v(k_x)$.
\end{rem}

\begin{NB}

Assume that $K_X=0$. 
For $\omega=tH$, $t>0$,
let $Z_{(0,\omega)}:{\bf D}(X) \to {\Bbb C}$ be a stability function
defined by
\begin{equation}
Z_{(0,\omega)}(E):=\langle e^{\omega \sqrt{-1}},v(E) \rangle,\;
E \in {\bf D}(E).
\end{equation}
Let ${\frak T}_{(0,\omega)}$ be the full subcategory 
of $\Coh(X)$ generated by torsion sheaves and torsion free
stable sheaves $E$ with
$Z_{(0,\omega)}(E) \in {\Bbb H} \cup {\Bbb R}_{<0}$.
Let ${\frak F}_{(0,\omega)}$ be the full subcategory 
of $\Coh(X)$ generated by torsion free
stable sheaves $E$ with
$-Z_{(0,\omega)}(E) \in {\Bbb H} \cup {\Bbb R}_{<0}$.
Let ${\frak A}_{(0,\omega)} (\subset {\bf D}(X))$ be the category generated by
${\frak T}_{(0,\omega)}$ and ${\frak F}_{(0,\omega)}[1]$.
If$(\omega^2) \ne 1$, then
$\sigma{(0,\omega)}:=({\frak A}_{(0,\omega)},Z_{(0,\omega)})$
is a stability condition.
${\frak A}_{(0,\omega)}$ is constant
on $(\omega^2) \ne 1$.
We set 
\begin{equation}
{\frak A}_{(0,\omega)}:=
\begin{cases}
{\frak A}^\mu, & (\omega^2)>1\\
{\frak A}, & (\omega^2)<1.
\end{cases}
\end{equation}
\begin{defn}
\begin{enumerate}
\item[(1)]
For $(\omega^2)>1$, we set
${\frak T}^\mu:={\frak T}_{(0,\omega)}$,
${\frak F}^\mu:={\frak F}_{(0,\omega)}$ and
${\frak A}^\mu:={\frak A}_{(0,\omega)}$.
\item[(2)]
For $(\omega^2)<1$, we set
${\frak T}:={\frak T}_{(0,\omega)}$,
${\frak F}:={\frak F}_{(0,\omega)}$ and
${\frak A}:={\frak A}_{(0,\omega)}$.
\end{enumerate}
\end{defn}
For $E \in {\frak F}^\mu$, 
we have an exact sequence
\begin{equation}
0 \to E_1 \to E \to E_2 \to 0
\end{equation} 
such that
\begin{enumerate}
\item[(1)]
 $E_1$ is generated by ${\cal O}_X$, and
\item[(2)]
 $E_2 \in {\frak F}^\mu$ satisfies
$\Hom({\cal O}_X,E_2)=0$, i.e.,
$E_2 \in {\frak F}$.
\end{enumerate}
For $E \in {\frak T}$, 
we have an exact sequence
$$
0 \to E_1 \to E \to E_2 \to 0
$$
such that $E_2$ is a successive extension of ${\cal O}_X$ and
$E_1 \in {\frak T}^\mu$.

As in \cite{MYY:2011:1},
$\Phi_{X \to X}^{{\cal E}^{\vee}[1]}:{\bf D}(X) \to
{\bf D}(X)$
induces an isomorphism
${\frak A} \to {\frak A}^\mu$
\begin{NB2}
${\cal O}_X$ are irreducible objects of ${\cal A}$.
${\cal E}_{|\{x \} \times X}[1]$ is an irreducible objects
of ${\cal A}$.
$k_x$ is not irreducible in ${\frak A}$:
\begin{equation}
0 \to F \to k_x \to  {\cal E}_{|\{x \} \times X}[1]
\to 0,
\end{equation}
where $F$ is a non-trivial extension
$$
0 \to {\cal O}_X \to F \to {\cal O}_X \to 0.
$$ 
\end{NB2}
and we have a commutative diagram
\begin{equation}
\begin{CD}
{\frak A} @>{\Phi_{X \to X}^{{\cal E}^{\vee}[1]}}>> {\frak A}^\mu\\
@V{Z_{(0,\omega)}}VV @VV{Z_{(0,\omega')}}V\\
{\Bbb C} @<<{\times (\omega^2)}< {\Bbb C}
\end{CD}
\end{equation}
where $\omega'=\omega/(\omega^2)$.
In particular, we get the following.
\begin{prop}
$\Phi_{X \to X}^{{\cal E}^{\vee}[1]}$ induces an isomorphism
\begin{equation}
{\cal M}_{(0,\omega)}\left(r,\eta+\frac{r}{2}K_X,-\frac{s}{2}\right)
\cong {\cal M}_{(0,\omega')}\left(s,\eta+\frac{s}{2}K_X,-\frac{r}{2}\right).
\end{equation}
\end{prop}

\begin{rem}
Since ${\cal O}_X[1]$ is an irrreducible object of
$\Phi_{X \to X}^{{\cal E}[1]}({\cal O}_X[1])$ is irreducible.
Indeed since $\dim H^2({\cal E}_{|\{ x \} \times X})=1$,
$\chi({\cal E}_{|\{ x \} \times X})=1$ implies
$H^1({\cal E}_{|\{ x \} \times X})=0$ for all $x \in X$.
Hence $R^2 p_{2*}({\cal E})=\Phi_{X \to X}^{{\cal E}}({\cal O}_X[2])$
is a line bundle. It is easy to see that it is ${\cal O}_X$.  
\end{rem}

\end{NB}

\subsection{The case where $r$ is odd}
For a numerically equivalence class $\xi$,
let $N_H(r,\xi,\frac{s}{2})$ be the moduli space
of stable sheaves $E$ of $\rk E=r$,
$c_1(E) \mod K_X =\xi$ and $\chi(E)=(r+s)/2$.
Then we have a morphism
$N_H(r,\xi,\frac{s}{2}) \to \Pic^{\xi}(X)$,
where $\Pic^{\xi}(X)$ is the subscheme of $\Pic(X)$
consisting of divisor classes $L$ with $L \mod K_X= \xi$.

\begin{prop}
If $r$ is odd, then
$N_H(r,\xi,\frac{s}{2}) \to \Pic^{\xi}(X)$ is a smooth morphism.
In particular $M_H(v)$ is smooth of $\dim M_H(v)=\langle v^2 \rangle+1$,
where $L \mod K_X =\xi$ and $v=(r,L,\frac{s}{2})$.
\end{prop}

\begin{proof}
We take $E \in M_H(v)$.
If $K_X \ne 0$, then
$\Pic(X)$ is smooth.
Moreover $E \otimes K_X \not \cong E$.
Hence $M_H(v)$ is smooth at $E$.
Assume that $K_X=0$. We note that
$\chr(k)=2$.
For the Fourier-Mukai transform in \eqref{eq:FM},
Proposition \ref{prop:FM} implies $\det \Phi_{X \to X}^{{\cal E}}(E)=-\det E$.
Thus 
we have a commutative diagram
\begin{equation}
\begin{CD}
{\bf D}(X) @>{\Phi_{X \to X}^{{\cal E}}}>> {\bf D}(X)\\
@V{\det}VV @VV{\det}V\\
\Pic(X)@>{D}>> \Pic(X)
\end{CD}
\end{equation}
where $D$ is the taking dual map.
Hence $\Phi_{X \to X}^{{\cal E}}$ preserves the structure
of fibration $\det$.
In particular, the smoothness of $\det$ is preserved under
  $\Phi_{X \to X}^{{\cal E}}$.
Replacing $E$ by $\Phi_{X \to X}^{{\cal E}}(E(nH))$,
we may assume that 
$E$ is locally free. Then the obstruction of infinitesimal lifting
lives in $H^2({\cal E}nd_0(E))$ by Lemma \ref{lem:obst}.
Since $(\rk E,\chr k)=1$, 
${\cal E}nd_0(E)$ is a direct summand of 
${\cal E}nd(E)$. 
Since $\Ext^2(E,E)=\Hom(E,E)^{\vee} \cong k$,
we get $H^2({\cal E}nd_0(E))=0$, which implies
$\det$ is smooth.
\begin{NB}
Let $\mathrm{tr}:{\bf R}{\cal H}om(E,E) \to {\cal O}_X$ 
be the trace map. 
For the homomorphism
$\iota:{\cal O}_X \to {\bf R}{\cal H}om(E,E)$ 
induced by the multiplication by constant,
$\mathrm{tr} \circ \iota$ is isomorphic. 
Hence  $\mathrm{tr}^i:\Ext^i(E,E) \to H^i({\cal O}_X)$
are surjective for all $i$.
\end{NB}
Since $H^1({\cal E}nd_0(E))=\ker \tr^1$ is the tangent space
of $M_H(v)$,
$\dim M_H(v)=\langle v^2 \rangle+1$.
\end{proof}

By using \cite[Cor. 4.5]{Y:Enriques},
we get a similar result to \cite{Y:twist1}.
\begin{prop}
Assume that $X$ is defined over a finite field ${\Bbb F}_q$.
If $r$ is odd, then
$\#M_H(v)({\Bbb F}_q)=\# \Hilb_X^{(\langle v^2 \rangle+1)/2}({\Bbb F}_q)$.
In particular $M_H(v)$ is irreducible.
\end{prop}

\begin{NB}
Assume that $E$is locally free.
Then $\ker \mathrm{tr}={\cal H}om(E,E)_0$
is a locally free sheaf and the obstruction of lifting
of $(E,\det E \cong L)$ belongs to
$H^2({\cal H}om(E,E)_0)$.
If $\gcd(\rk E,\chr(k))=1$, then 
$H^2({\cal H}om(E,E)_0) 
=\ker \mathrm{tr}^2$, however 
in general, $H^1({\cal O}_X) \to H^2({\cal H}om(E,E)_0)$
may not be injective.

If $\Pic^0(X)$ is smooth, then
the obstruction of lifting of $E$ itself belongs to
$\ker \mathrm{tr}^2$.
But if $\Pic^0(X)$ is not smooth, then we can only say that
the obstruction of lifting of $E$ itself belongs to
$\Ext^2(E,E)$.
\end{NB}

\begin{lem}\label{lem:obst}
For a locally free sheaf $E_0 \in N_H(r,\xi,\frac{s}{2})$,
the obstruction of infinitesimal lifting at $E_0$ with respect
to the morphism
 $N_H(r,\xi,\frac{s}{2}) \to \Pic(X)$ belongs to
$H^2({\cal E}nd_0(E_0))$, where
${\cal E}nd_0(E_0):=\ker({\cal E}nd(E_0) \overset{\tr}{\to}
{\cal O}_X)$ is the sheaf of trace free homomorphisms.
\end{lem}

\begin{proof}
Let $(R,{\frak m})$ be an Artinian local ring over $k$
with $R/{\frak m} \cong k$ and $(0 \ne) \epsilon \in {\frak m}$ satisfies
${\frak m}\epsilon=0$. We set $R':=R/k \epsilon$.
Assume that there is a locally free sheaf $E'$ on
$X \otimes_k {R'}$ and a line bundle $L$ on $X \otimes_k R$ such that
$L \otimes_R R'=\det E$.
We take an open covering $\{U_i \}_i$ of $X$ such that
$E'$ is trivial over $U_i$.
We set
$U_{ij}:=U_i \cap U_j$ and $U_{ijk}:=U_i \cap U_j \cap U_k$.  
Let 
$A_{ij}:{\cal O}_{U_{ij}\otimes_k R}^{\oplus r}
\to  {\cal O}_{U_{ij}\otimes_k R}^{\oplus r}$
be a lifting of the patching data $A_{ij}':
{\cal O}_{U_{ij}\otimes_k R'}^{\oplus r}  
\to  {\cal O}_{U_{ij}\otimes_k R'}^{\oplus r}$
of $E'$. 
For the induced map
$\det A_{ij}:
\det {\cal O}_{U_{ij}\otimes_k R}^{\oplus r} \to 
\det {\cal O}_{U_{ij}\otimes_k R}^{\oplus r}$,
$\det A_{ij} \equiv a_{ij} \mod k \epsilon$.
Hence $(\det A_{ij})a_{ij}^{-1}=1+b_{ij}\epsilon$
($b_{ij} \in H^0(U_{ij},{\cal O}_{U_{ij}})$).
We take $B_{ij} \in GL(r,{\cal O}_{U_{ij}})$  
such that $\tr(B_{ij})=b_{ij}$.
For $A_{ij}(I-B_{ij}\epsilon)$,
\begin{equation}
\begin{split}
\det(A_{ij}(I-B_{ij}\epsilon))=&\det A_{ij}
\det (I-B_{ij}\epsilon)\\
=& \det A_{ij}(1-b_{ij}\epsilon)
=(a_{ij}(1+b_{ij} \epsilon))(1-b_{ij}\epsilon)=a_{ij}.
\end{split}
\end{equation}
Replacing $A_{ij}$ by $A_{ij}(I-B_{ij}\epsilon)$,
we may assume that $\det A_{ij}=a_{ij}$.
Then $A_{ki}A_{jk}A_{ij}=I+B_{ijk}\epsilon$,
where $\{ B_{ijk} \}$ is a 2-cocycle of ${\cal E}nd_0(E_0)$. 
If $A_{ij}(I+C_{ij}\epsilon)$ also induces an extension of isomorphism
$\det E \cong L$, then there are isomorphisms
$\psi_i:\det {\cal O}_{U_i \otimes_k R}^{\oplus r}
\to \det {\cal O}_{U_i \otimes_k R}^{\oplus r}$ such that
$a_{ij}=\psi_j 
\det (A_{ij}(I+C_{ij}\epsilon) )\psi_i^{-1}$.
Since 
$a_{ij} \equiv \psi_j 
\det A_{ij}\psi_i^{-1} \equiv
\psi_j a_{ij}\psi_i^{-1} \mod k \epsilon$,
there is $\psi' \in H^0({\cal O}_{X \otimes_k {R'}})=R'$ such that
$(\psi_i) \otimes R' =\psi'_{|U_i \otimes R'}$.
Let $\psi \in R$ be an extension of $\psi'$. 
Replacing $\psi_i$ by $\psi^{-1} \psi_i$,
we may assume that $\psi_i=1+\lambda_i \epsilon$.
Thus 
we have 
$a_{ij}=(1+\lambda_j \epsilon) 
\det (A_{ij}(I+C_{ij}\epsilon) )(1+\lambda_i \epsilon)^{-1}$.
Then
$a_{ij}=a_{ij}(1+(\lambda_j+\tr C_{ij}-\lambda_i)\epsilon)$, which implies
$\tr C_{ij}+\lambda_j-\lambda_i=0$.
Let $\Lambda_i:{\cal O}_{U_i \otimes_k R}^{\oplus r} \to 
{\cal O}_{U_i \otimes_k R }^{\oplus r}$
be homomorphisms with $\tr \Lambda_i=\lambda_i$.
Then 
$$
(I+\Lambda_j \epsilon)A_{ij}(I+C_{ij}\epsilon)(I+\Lambda_i \epsilon)^{-1}
=A_{ij}(I+C_{ij}'\epsilon),\;
C_{ij}':=(A_{ij}^{-1}\Lambda_j A_{ij})+C_{ij}-\Lambda_i,
$$
$\tr C_{ij}'=\lambda_j+\tr C_{ij}-\lambda_i=0$ and
$$
\det (A_{ij}(I+C_{ij}'\epsilon))=
\det A_{ij}(1+(\lambda_j+\tr C_{ij}-\lambda_i)\epsilon)
=a_{ij}.
$$
Moreover for the coboundary map
$\partial$ of ${\cal E}nd_0(E_0)$, we have $\partial C_{ij}'=\partial C_{ij}$.
If $B_{ijk}+\partial C_{ij}'=0$ for some $C_{ij}' \in 
{\cal E}nd_0(E_0)_{|U_{ij}}$, then 
we get a lifting $E$ of $E'$ with $\det E \cong L$.
\begin{NB}
\begin{equation}
(A_{ki}(I+C_{ki}' \epsilon)) (A_{jk}(I+C_{jk}' \epsilon)) 
(A_{ij}(I+C_{ij}' \epsilon))
=I+(B_{ijk}+C_{ij}'+A_{ij}^{-1}C_{jk}'A_{ij}+A_{ik}^{-1}C_{ki}'A_{ik})\epsilon 
=I+(B_{ijk}+\partial C_{ijk}')\epsilon.
\end{equation}
\end{NB}
Therefore the obstruction lives in
$H^2({\cal E}nd_0(E_0))$.
\begin{NB}
Tangent space:
Let $A_{ij}:{\cal O}_{U_{ij}}^{\oplus r} \to {\cal O}_{U_{ij}}^{\oplus r}$
be the patching data of $E_0$.  
The first order deformation of $E_0$ is parametrized by
$A_{ij}(I+B_{ij}\epsilon)$ such that
$\psi_j\det(A_{ij}(I+B_{ij}\epsilon))\psi_i^{-1}=a_{ij}$,
where $\det A_{ij}=a_{ij}$.
Then $\psi_i \mod k \epsilon= \psi_{|U_i}$
($\psi \in H^0({\cal O}_X)=k$).
Replacing $\psi_i$ by $\psi^{-1}\psi_i$, we have
$\psi_i=1+\lambda_i \epsilon$.
Then $\tr B_{ij}+\lambda_j-\lambda_i=0$.
Thus $\tr B_{ij}$ is zero in $H^1({\cal O}_X)$.
If we replace the trivialization
$I+P_i \epsilon:
{\cal O}_{U_{i}}^{\oplus r} \to {\cal O}_{U_{i}}^{\oplus r}$,
then $B_{ij}$ is replaced by
$B_{ij}'=-P_i+B_{ij}+A_{ij}^{-1}P_j A_{ij}$.
Hence $B_{ij} \mod \partial {\cal E}nd(E_0)$
is the data of the infinitesimal deformation. 
\end{NB}
\end{proof}

\begin{prop}
For $E \in M_H(r,L,\frac{s}{2})$, 
Zariski tangent space is 
$\ker(\Ext^1(E,E) \overset{\tr^1}{\to} H^1({\cal O}_X))$ and
$$
\dim M_H(r,L,\tfrac{s}{2})\geq (L^2)-rs+1.
$$
\end{prop}

\begin{proof}
We only treat the case where $K_X=0$.
The first claim is obvious by the definition of 
$M_H(r,L,\frac{s}{2})$.
For the second claim, we may assume that $E$ is locally free by using
a Fourier-Mukai transform.
In this case, we have an exact sequence
$$
0 \to \ker \tr^1 \to \Ext^1(E,E) \to H^1({\cal O}_X) \to
H^2({\cal E}nd_0(E)) \to 0.
$$
Hence 
$$
\dim M_H(r,L,\tfrac{s}{2}) \geq
\dim \ker \tr^1-\dim H^2({\cal E}nd_0(E))=(\langle v(E)^2 \rangle+2)-1=
(L^2)-rs+1.
$$
\end{proof}

\begin{NB}
Assume that $X$ is an abelian surface.
There is an example of smooth moduli space. 
\begin{lem}[{\cite[Prop. 4.1]{Y:Nagoya}}]
If $(r,\chi(L))=1$, then
$M_H(r,c_1(L),a) \cong M_H(r,L,a) \times \Pic^0(X)$.
In particular, $M_H(r,L,a)$ is smooth.
\end{lem}

Assume that $E \in M_H(r,L,a)$ is locally free.
Then $H^2({\cal E}nd_0(E))=0$ by the surjectivity of
$\tr^1$.

For $M_H(2,0,-1)$, we have an isomorphism
$M_H(2,0,-1) \to X \times \Hilb_{\widehat{X}}^2$
such that 
$\det: M_H(2,0,-1) \to \Pic^0(X)$ corresponds to
$X \times \Hilb_{\widehat{X}}^2 \to \Hilb_{\widehat{X}}^2  
\to \widehat{X}$, where $\widehat{X}=\Pic^0(X)$.
Hence $\det$ is not smooth, if $\chr(k)=2$ by
\cite{Sch}.
In particular $\tr^1$ is not surjective in general.

Assume that $\chr(k) \mid r$. Let $E$ be a simple semi-homogeneous
bundle on an abelian variety $X$.
\begin{enumerate}
\item
If $\dim X=1$, then
$H^2({\cal E}nd_0(E))=0$, which implies that
$\tr^1$ is isomorphic.
Hence $H^0({\cal O}_X) \cong H^1({\cal E}nd_0(E)) \cong k$
and $H^0({\cal E}nd_0(E)) \cong H^0({\cal E}nd(E))\cong k$.
\item
If $\dim X=2$, then
$\tr^1$ is the dual of 
$H^1({\cal O}_X) \to H^1({\cal E}nd(E))$ which is not isomorphic
by \cite[Prop. 5.9]{Mu0}.
In this case,
$H^2({\cal E}nd_0(E)) \ne 0$ and
$\ker \tr^1 \ne 0$. 
\end{enumerate}

\end{NB}

\end{document}